\documentclass[12pt]{iopart}

\usepackage{iopams}
\usepackage{amsthm}
\usepackage{citesort}

\newcommand{\Rd}{\mathbb{R}^d}
\newcommand{\trac}{\mathbf{t}}
\newcommand{\xx}{\mathbf{x}}
\newcommand{\uu}{\mathbf{u}}
\newcommand{\uut}{\widetilde{\uu}}
\newcommand{\ut}{\widetilde{u}}
\newcommand{\vt}{\widetilde{v}}
\newcommand{\uer}{\ut_{E,r}}
\newcommand{\uet}{\ut_{E,\theta}}
\newcommand{\utr}{\ut_r}
\newcommand{\utt}{\ut_{\theta}}
\newcommand{\Xt}{\widetilde{X}}
\newcommand{\Yt}{\widetilde{Y}}
\newcommand{\ff}{\mathbf{f}}
\newcommand{\FF}{\mathbf{F}}
\newcommand{\nn}{\mathbf{n}}
\newcommand{\Atens}{\mathbb{A}}
\newcommand{\Btens}{\mathbf{B}}
\newcommand{\Ctens}{\mathbb{C}}
\newcommand{\II}{\mathbf{I}}
\newcommand{\PP}{\mathbf{P}}
\newcommand{\IIII}{\mathbb{I}}
\newcommand{\eps}{\bvarepsilon}
\newcommand{\sss}{\bsigma}
\newcommand{\ssst}{\widetilde{\sss}}
\newcommand{\st}{\widetilde{\sigma}}
\newcommand{\strr}{\st_{E,rr}}
\newcommand{\strt}{\st_{E,r\theta}}
\newcommand{\Sym}{\mathrm{Sym}(\Rd)}
\newcommand{\SymSym}{\mathrm{Sym}(\Sym)}
\newcommand{\Frr}{F_{rr}}
\newcommand{\Frt}{F_{r\theta}}
\newcommand{\eqand}{\quad \mathrm{and} \quad}
\newcommand{\ds}{\displaystyle}
\newcommand{\dd}{\, \mathrm{d}}
\newcommand{\dfrac}{\ds\frac}

\newtheorem{theorem}{Theorem}[section]
\newtheorem{lemma}[theorem]{Lemma}

\theoremstyle{definition}
\newtheorem{definition}[theorem]{Definition}

\theoremstyle{remark}
\newtheorem{remark}[theorem]{Remark}

\begin{document}

\title[Exact determination of the volume of an inclusion]%
{Exact determination of the volume of an inclusion in a body having %
constant shear modulus}
\author{Andrew E Thaler and Graeme W Milton} 
\address{Department of Mathematics, University of Utah, Salt Lake City, 
UT 84112, USA}
\ead{\mailto{thaler@math.utah.edu}, \mailto{milton@math.utah.edu}}

%%%%%%%%%%%%%%%%%%%%%%%%%%%%%%%%%%%%%%%%%%%%%%%%%%%%%%%%%%%%%%%%%%%%%%%%
%%%%%%%%%%%%%%%%%%%%%%%%%%%%%%%%%%%%%%%%%%%%%%%%%%%%%%%%%%%%%%%%%%%%%%%%

\begin{abstract}
    We derive an exact formula for the volume fraction of an inclusion 
    in a body when the inclusion and the body are linearly elastic 
    materials with the same shear modulus.  Our formula depends on an 
    appropriate measurement of the displacement and traction around the 
    boundary of the body.  In particular, the boundary conditions 
    around the boundary of the body must be such that they mimic the 
    body being placed in an infinite medium with an appropriate 
    displacement applied at infinity.  
\end{abstract}

\submitto{\IP}
\maketitle

%%%%%%%%%%%%%%%%%%%%%%%%%%%%%%%%%%%%%%%%%%%%%%%%%%%%%%%%%%%%%%%%%%%%%%%%
%%%%%%%%%%%%%%%%%%%%%%%%%%%%%%%%%%%%%%%%%%%%%%%%%%%%%%%%%%%%%%%%%%%%%%%%

\section{Introduction}\label{sec:EVF_introduction}

A fundamental and interesting problem in the study of materials is the
estimation of the volume fraction occupied by an inclusion $D$ in a body
$\Omega$. Although the volume fraction could be determined by weighing
the body, the densities of the materials may be close or unknown or
weighing the body may be impractical. Because of this, many methods have
been developed which utilize measurements of certain fields around
$\partial \Omega$ to derive bounds on the volume fraction
$|D|/|\Omega|$, where $|U|$ is the Lebesgue measure of the set $U$
\cite{McPhedran:1982:ESI,Phan-Thien:1982:PUB,McPhedran:1990:ITP,%
Kang:1997:ICP,Alessandrini:1998:ICP,Cherkaeva:1998:IBM,%
Ikehata:1998:SEI,Alessandrini:1999:OSE,Alessandrini:2002:DIE,%
Alessandrini:2002:DCE,Capdeboscq:2003:OAE,Alessandrini:2004:DIEb,%
Capdeboscq:2008:IHS,Morassi:2009:DCI,Beretta:2011:SEE,Kang:2012:SBV,%
Milton:2012:BVF,Milton:2012:UBE,Kang:2013:BSI,Kang:2013:BVF3d,%
Thaler:2013:BVI,Kang:2014:BSS,Thaler:2014:BVI}. In this paper, we 
show that under certain circumstances the volume fraction $|D|/|\Omega|$
can be computed exactly from a single appropriate boundary measurement
around $\partial \Omega$. We note that many of the results in the
literature (and our results in this paper) can also be applied when
$\Omega$ contains a two-phase composite with microstructure much smaller
than the dimensions of $\Omega$.  

We consider an inclusion $D$ in a body $\Omega$ (or a two-phase
composite inside $\Omega$), where $\Omega$ is a subset of $\Rd$ ($d = 2$
or $3$).  We assume that the inclusion and body are filled with
linearly elastic materials with the same shear modulus $\mu$ and Lam\'e
moduli $\lambda_1$ and $\lambda_2$, respectively. Our goal is to
determine the volume fraction occupied by the inclusion, namely
$|D|/|\Omega|$, in terms of a measurement of the displacement and
traction around $\partial \Omega$. The boundary conditions around
$\partial \Omega$ are taken to be such that they mimic the body $\Omega$
being placed in an infinite medium with a suitable field at infinity.
The starting point for our result is based on an exact relation due to
Hill \cite{Hill:1963:EPR}, which we now describe.

One of the most important problems in the study of composite materials
is the determination of effective moduli given information about the
local moduli --- see, e.g., the work by Hashin \cite{Hashin:1983:ACM}
and the book by Milton \cite{Milton:2002:TOC} (Chapters 1 and 2 in
particular). In general, it is extremely difficult (if not impossible)
to determine effective parameters exactly, even if the microgeometry of
the composite is known and relatively uncomplicated. However, many
useful approximation techniques and bounds on effective properties of
composites have been derived in the literature --- see the book by
Milton \cite{Milton:2002:TOC} and the references therein for a vast
collection of such results. Surprisingly, there are several
circumstances in which exact links between effective moduli (or exact
formulas for the moduli themselves) can be derived regardless of the
complexity of the microstructure; such links are known as \emph{exact
relations}.  

Exact relations exist for a variety of problems including elasticity and
coupled problems such as thermoelasticity, thermoelectricity,
piezoelectricity, thermo-piezo-electricity, and others --- see the
review articles by Dvorak and Benveniste \cite{Dvorak:1997:OMI} and
Milton \cite{Milton:1997:CMM}, the work by Grabovsky, Milton, and Sage
\cite{Grabovsky:2000:ERE}, the works by Hegg
\cite{Hegg:2012:ERL,Hegg:2013:LBE}, and the references therein for
summaries of numerous previous and current results on exact relations.  

Perhaps even more surprising than the existence of exact relations is
the existence of a general mathematical theory of exact relations,
developed by Grabovsky, Milton, and Sage \cite{Grabovsky:1998:EREa,%
Grabovsky:1998:EREb,Grabovsky:1998:ERC,Grabovsky:1998:ROP,%
Grabovsky:2000:ERE}, that allows us to determine all of the above
mentioned exact relations and many more. For example, Hegg applied this
general theory to the study of fiber-reinforced elastic composites
\cite{Hegg:2012:ERL,Hegg:2013:LBE}.

Rather than study the general theory, we focus on a specific exact
relation derived by Hill \cite{Hill:1963:EPR,Hill:1964:TMP}. In
particular, Hill considered a two-phase composite material consisting of
two homogeneous and isotropic phases with the same shear modulus $\mu$
but different Lam\'e moduli $\lambda_1$ and $\lambda_2$. Hill proved
that such a composite is macroscopically elastically isotropic with
shear modulus $\mu$ and effective Lam\'e modulus $\lambda_*$; he also
derived an exact formula for $\lambda_*$ that holds regardless of the
complexity of the microgeometry.  His
derivation of this formula \cite{Hill:1963:EPR,Hill:1964:TMP} provides
the starting point of our work in this paper. 

We begin by assuming that the body $\Omega$ is embedded in an infinite
medium with Lam\'e modulus $\lambda_E$ and shear modulus $\mu$ (we take
$\lambda_E = \lambda_2$ for simplicity) and that a displacement $\uu =
\nabla g$ is applied at infinity. Using a method similar to Hill's
derivation of $\lambda_*$, we derive a formula for $|D|/|\Omega|$ in
terms of a measurement of the displacement around $\partial \Omega$, the
(known) parameters $\lambda_1, \lambda_2$, and $\mu$, and the (known)
function $g$. In order to make the situation more practical, we derive a
certain nonlocal boundary condition that can be applied to $\partial
\Omega$ that forces the body to behave as if it actually were embedded
in an infinite medium with Lam\'e modulus $\lambda_2$, shear modulus
$\mu$, and an applied displacement $\uu = \nabla g$ at infinity. This
nonlocal boundary condition couples the measurements of the traction and
displacement around $\partial \Omega$.  

Nonlocal boundary conditions which mimic infinite media similar to the
one mentioned above are common tools used in the numerical solution of
PDEs (and ODEs) in infinite domains --- see, e.g., the review article by
Givoli \cite{Givoli:1991:NRB} and references therein for examples
specific to scattering problems, the work by Han and Wu on the Laplace
equation \cite{Han:1985:AIB} and elasticity equations
\cite{Han:1992:AEB}, the work by Lee, Caflisch, and Lee on the
elasticity equations \cite{Lee:2006:EAB}, and references therein.  

As discussed above, our formula for the volume fraction $|D|/|\Omega|$
holds as long as the body $\Omega$ is embedded in an infinite medium
with an applied displacement $\uu = \nabla g$ at infinity. In
Section~\ref{sec:finite_medium} we derive a nonlocal boundary condition
such that if this boundary condition is applied to $\partial \Omega$ the
solution inside $\Omega$ will be equal to the restriction to $\Omega$ of
the solution to the infinite problem. Our boundary condition depends on
the function $g$ and on the exterior Dirichlet-to-Neumann map on
$\partial \Omega$ (which, when the body $\Omega$ is absent, maps the
displacement on $\partial \Omega$ to the traction on $\partial \Omega$
when no fields are applied at infinity). Thus it is closely related to
the boundary condition of Han and Wu \cite{Han:1985:AIB,Han:1992:AEB}
and Bonnaillie-No{\"e}l, Dambrine, H{\'e}rau, and Vial
\cite{Bonnaillie-Noel:2013:ACL} --- see Section~\ref{sec:finite_medium}
for complete details.  

The rest of this paper is organized as follows. In
Section~\ref{sec:elasticity} we briefly review the linear elasticity
equations and relevant results from homogenization theory. 
Next, in Section~\ref{sec:evf} we derive a formula that gives the exact
volume fraction of an inclusion in a body when the inclusion and the
body have the same shear modulus $\mu$ and the body is embedded in an
infinite medium with shear modulus $\mu$. We discuss the nonlocal
boundary condition relevant to our problem in
Section~\ref{sec:finite_medium} so we can focus on a (more realistic)
finite domain. Finally, in Section~\ref{sec:2D_example} we present the
analytical expression of the nonlocal boundary condition in the
particular case when $\Omega$ is a disk in $\mathbb{R}^2$ --- this
expression was first derived by Han and Wu \cite{Han:1985:AIB,%
Han:1992:AEB}. A complete derivation of our nonlocal boundary condition
is given in work by one of the authors of this paper
\cite{Thaler:2014:BVI}.

%%%%%%%%%%%%%%%%%%%%%%%%%%%%%%%%%%%%%%%%%%%%%%%%%%%%%%%%%%%%%%%%%%%%%%%%
%%%%%%%%%%%%%%%%%%%%%%%%%%%%%%%%%%%%%%%%%%%%%%%%%%%%%%%%%%%%%%%%%%%%%%%%

\section{Elasticity}\label{sec:elasticity}

We begin by recalling a few definitions from Hegg's work
\cite{Hegg:2012:ERL,Hegg:2013:LBE}. Let $d = 2$ or $3$ be the dimension
under consideration; then $\Sym$ is the set of all symmetric linear
mappings from $\Rd$ to itself, i.e.,
\begin{equation*}
	\Sym \equiv \left\{\Btens \in \Rd\otimes\Rd \,|\, \Btens 
	    = \Btens^T\right\}.
\end{equation*}
Similarly, the set $\SymSym$ is defined as the set of symmetric linear
mappings from $\Sym$ to itself. If $\Atens \in \SymSym$ and $\Btens \in
\Sym$, then $\Atens:\Btens \in \Sym$ with elements
\begin{equation}\label{eq:A_tensor_B}
	(\Atens:\Btens)_{ij} = A_{ijkl}B_{kl},
\end{equation}
where here and throughout the paper we use the Einstein summation 
convention that repeated indices are summed from $1$ to $d$.  
 
Consider a linearly elastic body which is either a periodic composite
material with unit cell $\Omega \subset \Rd$ or which occupies an open,
bounded set $\Omega \subset \Rd$. Let $\uu(\xx)$, $\eps(\xx)$, and
$\sss(\xx)$ denote the displacement, linearized strain tensor, and
Cauchy stress tensor, respectively, at the point $\xx \in \Omega$. Then
$\uu \in \Rd$ while $\eps$ and $\sss$ belong to $\Sym$. By Hooke's Law,
the stress and strain tensor are related through the linear constitutive
relation
\begin{equation}\label{eq:constitutive}
	\sss(\xx) = \Ctens(\xx) : \eps(\xx),
\end{equation}
where $\Ctens \in \SymSym$ is the elasticity (or stiffness) tensor. We
assume $\Ctens$ is elliptic for all $\xx \in \Omega$, i.e., there are
positive constants $a$ and $b$ such that
\begin{equation*}
	\Btens: \left[\Ctens(\xx) : \Btens'\right] 
	\le a \|\Btens\|\|\Btens'\| 
	\eqand \Btens:\left[\Ctens(\xx):\Btens\right] \ge b \|\Btens\|^2
\end{equation*}
for all $\Btens$, $\Btens' \in \Sym$ and where $\|\Btens\|^2 =
\frac{1}{2}(\Btens:\Btens)$. If there are no body forces present, then
at equilibrium the elasticity equations are
\begin{equation}\label{eq:elasticity_equations}
    \fl\nabla \cdot \sss(\xx) = 0, \quad 
    \eps(\xx) = \dfrac{1}{2}\left[\nabla \uu(\xx) 
        + \nabla \uu(\xx)^{T}\right], 
    \eqand \eps(\xx) = \Ctens(\xx) : \sss(\xx);
\end{equation}
see, e.g., Chapter 2 of the book by Milton \cite{Milton:2002:TOC}. 

If the composite is locally isotropic, 
the local elasticity tensor takes the form
\begin{equation*}%\label{eq:isotropic_C}
	\Ctens(\xx) = \lambda(\xx) \II\otimes\II + 2\mu(\xx)\IIII,
\end{equation*}
where $\lambda$ is the Lam\'e modulus, $\mu$ is the shear modulus, $\II
\in \Sym$ is the second-order identity tensor with elements $I_{ij} =
\delta_{ij}$ (where $\delta_{ij}$ is the Kronecker delta which is $1$ if
$i = j$ and $0$ otherwise), and $\IIII \in \SymSym$ is the fourth-order
identity tensor which maps an element in $\Sym$ to itself under
contraction \cite{Milton:2002:TOC}. In this case, Hooke's Law
\eref{eq:constitutive} reduces to
\begin{eqnarray}
	\mathcal{S}\uu(\xx) \equiv \sss(\xx) 
	&= \lambda(\xx)\Tr\left[\eps(\xx)\right]\II + 2\mu(\xx)\eps(\xx) 
	    \label{eq:isotropic_Hooke_eps} \\
	&= \lambda(\xx)\left[\nabla\cdot\uu(\xx)\right]\II 
	    + \mu(\xx)\left[\nabla\uu(\xx) + \nabla\uu(\xx)^T\right], 
	    \label{eq:isotropic_Hooke_u}
\end{eqnarray}
where $\mathcal{S}:\Rd \rightarrow \Sym$ is the linear stress operator
that maps the displacement $\uu$ to the stress $\sss$ (note that
$\mathcal{S}$ itself depends on $\xx$ through $\lambda(\xx)$ and
$\mu(\xx)$). 

The bulk modulus, Young's modulus, and the Poisson ratio are related to
$\lambda$ and $\mu$ by
\begin{equation*}
    \kappa = \lambda + \dfrac{2\mu}{d}, 
    \quad E = \dfrac{2\mu(d\lambda + 2\mu)}{(d-1)\lambda + 2\mu}, 
    \eqand \nu = \dfrac{\lambda}{(d-1)\lambda + 2\mu},
\end{equation*}
respectively \cite{Milton:2002:TOC}. Throughout this paper we assume
\cite{Ammari:2007:PMT}
\begin{equation}\label{eq:assumption_lambda_mu}
    \mu(\xx) > 0 \eqand d\lambda(\xx) + 2\mu(\xx) > 0.
\end{equation}

We define the average of a tensor-valued function $\mathbf{M}(\xx)$ over
a set $\mathcal{M} \subset \Rd$ by
\begin{equation}\label{eq:avg_definition_evf}
	\langle\mathbf{M}\rangle_{\mathcal{M}}\equiv\dfrac{1}{|\mathcal{M}|}
	    \int_{\mathcal{M}} \mathbf{M}(\xx) \dd\xx,
\end{equation}
where $|\mathcal{M}|$ denotes the Lebesgue measure of the set
$\mathcal{M}$. The effective elasticity tensor $\Ctens_*$ is defined at
sample points $\xx \in \Omega$ through
\begin{equation}\label{eq:effective_def}
	\langle\sss\rangle_{\Omega'(\xx)} = \Ctens_*(\xx)\langle\eps\rangle_{\Omega'(\xx)},
\end{equation}
where $\Omega'(\xx)$ is a suitably chosen representative volume element
centered at $\xx$. 

%%%%%%%%%%%%%%%%%%%%%%%%%%%%%%%%%%%%%%%%%%%%%%%%%%%%%%%%%%%%%%%%%%%%%%%%
%%%%%%%%%%%%%%%%%%%%%%%%%%%%%%%%%%%%%%%%%%%%%%%%%%%%%%%%%%%%%%%%%%%%%%%%

\section{Exact Volume Fraction}\label{sec:evf}

In this section we derive a formula that gives the exact volume fraction
occupied by an inclusion in a body, where our formula depends on a
boundary measurement of the displacement. 
Let $D$ and $\Omega$ be open,
bounded sets in $\Rd$ with $D\subset\Omega$. Suppose $\Rd$ is filled
with a linearly elastic, locally isotropic material with constant shear
modulus $\mu$ and Lam\'e modulus
\begin{equation}\label{eq:my_lambda}
    \lambda(\xx) = \lambda_1 \chi_D(\xx) + 
        \lambda_2 \chi_{\mathbb{R}^d \setminus \overline{D}}(\xx).
\end{equation}
Since the material is locally isotropic, the elasticity tensor is
\begin{equation}\label{eq:my_C}
    \Ctens(\xx) = \lambda(\xx)\II\otimes\II + 2\mu\IIII.
\end{equation}
We can write $\mathcal{S}\uu(\xx)$ from \eref{eq:isotropic_Hooke_u} as
\begin{equation}\label{eq:S_j}
	\fl\mathcal{S}\uu(\xx) =
		\cases{\mathcal{S}_1\uu(\xx) \equiv 
			    \lambda_1\left[\nabla\cdot\uu(\xx)\right] \II + 
			    \mu\left[\nabla\uu(\xx) + \nabla\uu(\xx)^T\right]
			    &for $\xx \in D$,\\
			\mathcal{S}_2\uu(\xx) \equiv 
			    \lambda_2\left[\nabla\cdot\uu(\xx)\right] \II + 
			    \mu\left[\nabla\uu(\xx) + \nabla\uu(\xx)^T\right]
			    &for $\xx \in \mathbb{R}^d \setminus\overline{D}$.}
\end{equation}
According to the elasticity equations in
\eref{eq:elasticity_equations}, the displacement $\uu$ satisfies
\begin{equation}\label{eq:full_displacement_problem}
    \cases{\mathcal{L}_1\uu = 0 &in $D$,\\
        \mathcal{L}_2\uu = 0 &in $\mathbb{R}^d \setminus 
            \overline{D}$,\\
        \uu, \sss\cdot\nn_D = \left(\mathcal{S}\uu\right)\cdot\nn_{D} 
            &continuous across $\partial D$, \\ 
        \uu-\ff = \mathcal{O}(|\xx|^{1-d}) &as 
            $|\xx|\rightarrow\infty$,}
\end{equation}
where $\mathcal{L}_j\uu = -(\lambda_j+\mu)\nabla(\nabla\cdot\uu) -
\mu\Delta\uu$ (for $j=1$, $2$) is the Lam\'e operator, $\nn_{D}$ is the
outward unit normal vector to $\partial D$, $\sss = \mathcal{S}\uu$ is
the stress tensor associated with $\uu$, and the function $\ff = \nabla
g$ is given and satisfies $\mathcal{L}_2 \ff = \mathcal{L}_2\nabla g =
0$ for all $\xx \in \Rd$. To avoid possible technical complications we
assume that $g$ is at least three times continuously differentiable in
$\mathbb{R}^d$. The function $\ff$ represents the ``displacement at
infinity''; perhaps the simplest example of such a function is $\ff(\xx)
= \xx$, in which case $g = \frac{1}{2}(\xx\cdot\xx) +\mathrm{constant}$.
As shown in Chapters~9~and~10 of the book by Ammari and Kang
\cite{Ammari:2007:PMT}, there exists a unique solution $\uu$ to
\eref{eq:full_displacement_problem} if $D$ is a Lipschitz domain. 

Following Hill's work \cite{Hill:1963:EPR,Hill:1964:TMP},
we assume there is a continuously differentiable potential $\phi$ such
that $\uu = \nabla \phi$. In particular, we assume $\phi$ and $\nabla
\phi$ are continuous across $\partial D$ (by
\eref{eq:full_displacement_problem}, $\uu = \nabla \phi$ must be
continuous across $\partial D$). Also, for $i$, $j = 1, \ldots, d$, we
have
\begin{equation}\label{eq:nnp_nng}
	(\nabla\nabla \phi)_{ij} = 
	    \frac{\partial^2\phi}{\partial x_i \partial x_j};
\end{equation}
note that the matrix $\nabla\nabla\phi$ is symmetric in each phase. We
only assume that $\phi$ and $\nabla \phi$ are continuous across
$\partial D$ (indeed, as shown by Hill \cite{Hill:1961:DRM},
$\partial^2 \phi/\partial x_i \partial x_j$ is discontinuous across
$\partial D$). Then from \eref{eq:elasticity_equations} we have
\begin{equation}\label{eq:potential_strain}
    \eps = \dfrac{1}{2}\left(\nabla\uu + \nabla\uu^T\right) 
         = \dfrac{1}{2}\left[\nabla\nabla\phi 
            + (\nabla\nabla\phi)^T\right]
         = \nabla \nabla \phi.
\end{equation}
From \eref{eq:nnp_nng} and \eref{eq:potential_strain} we have
$\Tr(\eps) = \Tr(\nabla\nabla \phi) = \Delta \phi$, where $\Delta =
\nabla\cdot\nabla = \partial^2/\partial x_i\partial x_i$ is the
Laplacian. Then \eref{eq:isotropic_Hooke_eps} and
\eref{eq:potential_strain} imply
\begin{equation}\label{eq:potential_stress}
    \sss(\xx) = \Ctens(\xx):\eps(\xx) 
        = \lambda(\xx)\Delta\phi\II + 2\mu\nabla\nabla\phi.
\end{equation}
Finally, for $j = 1$ and $j = 2$ we have
\begin{eqnarray}
    \mathcal{L}_j\uu 
        &= -(\lambda_j+\mu)\nabla\left(\nabla\cdot\nabla\phi\right) - 
            \mu\Delta\left(\nabla\phi\right) \nonumber \\
        &= -(\lambda_j+\mu)\nabla\left(\Delta\phi\right) - 
            \mu\nabla\left(\Delta\phi\right) \nonumber \\
        &= -(\lambda_j+2\mu)\nabla(\Delta\phi).\label{eq:potential_L}
\end{eqnarray}

By assumption, we have
\[
    0 = \mathcal{L}_2\ff = -(\lambda_2+2\mu)\nabla(\Delta g)
\]
for all $x\in \Rd$, so $\nabla(\Delta g) = 0$ for all $\xx \in \Rd$.
Then we must have $\Delta g = C_g \ne 0$ for all $\xx \in \Rd$, where
$C_g$ is a constant. (The constant $C_g$ is known since $g$ is known; we
will see later why we must take $C_g \ne 0$.) Thus the function $g$ must
be chosen so that $g = (C_g/2)\xx\cdot\xx + g_h$, where $g_h$ is
harmonic in $\Rd$. This implies that $g$ is infinitely differentiable in
$\mathbb{R}^d$ \cite{Evans:2010:PDE}.

Recalling that $\uu = \nabla \phi$ and $\ff = \nabla g$, we see that
\eref{eq:potential_L} implies that \eref{eq:full_displacement_problem}
becomes \begin{equation}\label{eq:potential_full_displacement_problem}
    \cases{\nabla(\Delta\phi) = 0 &in $D$ and 
            $\Rd \setminus \overline{D}$, \\
        \nabla\phi,\ \sss\cdot\nn_D = (\mathcal{S}\nabla\phi)\cdot\nn_D 
            &continuous across $\partial D$, \\
        \nabla\phi-\nabla g = \mathcal{O}(|x|^{1-d}) 
            &as $|\xx| \rightarrow \infty$,}
\end{equation}
where $\sss = \mathcal{S}\nabla\phi$ is given in
\eref{eq:potential_stress}.

%%%%%%%%%%%%%%%%%%%%%%%%%%%%%%%%%%%%%%%%%%%%%%%%%%%%%%%%%%%%%%%%%%%%%%%%

\subsection{Behavior of $\Delta\phi$}\label{subsec:behavior}

In this section we study the behavior of $\Delta \phi$. Recall that we
assume $\phi$ to be at least continuously differentiable in $\Rd$; in
particular, this implies that $\phi$ and $\uu = \nabla \phi$ are
continuous across $\partial D$ (see
\eref{eq:potential_full_displacement_problem}). If $\ff$ is smooth the
solution $\uu$ to \eref{eq:full_displacement_problem} is smooth in
$\mathbb{R}^d \setminus \overline{D}$ and in $D$
\cite[equation~(10.2)]{Ammari:2007:PMT} (although it is only continuous
across $\partial D$).  
   
Since $\phi$ is smooth in $D$ and $\mathbb{R}^d \setminus \overline{D}$,
\eref{eq:potential_full_displacement_problem} implies that $\Delta\phi$
is constant in each phase, i.e., 
\begin{equation}\label{eq:phi_constant}
    \Delta\phi = 
    \cases{C_1 &in $D$,\\
        C_2 &in $\Rd \setminus \overline{D}$.}
\end{equation}

Recall from \eref{eq:elasticity_equations} that $\nabla \cdot \sss = 0$
in $\Rd$. By \eref{eq:potential_stress}, this becomes 
\begin{eqnarray}
    0 &= \nabla\cdot\sss(\xx) \nonumber \\
        &= \nabla \cdot\left[\lambda(\xx) \Delta \phi(\xx) \II +
            2\mu\nabla\nabla\phi(\xx)\right] \nonumber\\
        &= \nabla\left[\lambda(\xx)\Delta\phi(\xx)\right] + 
            2\mu\nabla\cdot\nabla\nabla\phi(\xx)\nonumber\\
        &= \nabla\left[\lambda(\xx)\Delta\phi(\xx)\right] 
            + 2\mu\Delta\left[\nabla\phi(\xx)\right] \nonumber\\
        &= \nabla\left\{\left[\lambda(\xx) + 2\mu\right]\Delta\phi(\xx)
            \right\}. \label{eq:grad_zero}      
\end{eqnarray}
This implies that 
\begin{equation}\label{eq:q_constant}
    \left[\lambda(\xx) + 2\mu\right]\Delta\phi(\xx) = C \quad 
    \Leftrightarrow \quad
    \Delta\phi(\xx) = \frac{C}{\lambda(\xx) + 2\mu}
\end{equation}
almost everywhere in $\mathbb{R}^d$, where $C$ is a constant 
\cite[Chapter~5]{Milton:2002:TOC}. 
Then, due to
\eref{eq:phi_constant} and \eref{eq:q_constant}, we have
\begin{equation}\label{eq:q_phi}
    \Delta\phi = 
    \cases{C_1 = \dfrac{C}{\lambda_1+2\mu} &in $D$,\\
        C_2 = \dfrac{C}{\lambda_2+2\mu} &in 
        $\Rd \setminus\overline{D}$.}
\end{equation}     

By \eref{eq:potential_full_displacement_problem}, we have $\nabla \phi
- \nabla g \rightarrow 0$ as $|\xx|\rightarrow \infty$; thus $\Delta
\phi - \Delta g \rightarrow 0$ as $|\xx|\rightarrow \infty$. Since
$\Delta g = C_g$, $\Delta \phi \rightarrow C_g$ as $|\xx|\rightarrow
\infty$. Since $\lambda(\xx) = \lambda_2$ for large enough $\xx$, we
take the limit of \eref{eq:q_constant} and find that
\begin{equation}\label{eq:C_E}
    C = \ds\lim_{|\xx|\rightarrow\infty}\left\{\left[\lambda(\xx) +2\mu
        \right]\Delta\phi(\xx)\right\} 
    = (\lambda_2+2\mu)C_g.
\end{equation}
Finally, \eref{eq:my_lambda}, \eref{eq:q_constant}, and \eref{eq:C_E}
imply that 
\begin{equation}\label{eq:Cs}
    \Delta\phi = \nabla \cdot \uu = 
    \cases{C_1 = \left(\dfrac{\lambda_2+2\mu}{\lambda_1+2\mu}\right)C_g 
            &in $D$,\\
        C_g &in $\Rd\setminus\overline{D}$.}
\end{equation}

%%%%%%%%%%%%%%%%%%%%%%%%%%%%%%%%%%%%%%%%%%%%%%%%%%%%%%%%%%%%%%%%%%%%%%%%

\subsection{Main Result}\label{subsec:main_result}

The divergence theorem and \eref{eq:Cs} imply
\begin{equation*}
    \ds\int_{\partial \Omega} \uu\cdot\mathbf{n}_{\Omega} \dd S 
    = \int_{D} \nabla\cdot\uu \dd\xx + 
        \int_{\Omega\setminus \overline{D}} \nabla\cdot\uu \, d\xx
    = (C_1-C_2)|D| + C_2|\Omega|.
\end{equation*}
Therefore the volume fraction of the inclusion is given by the formula
\begin{equation}\label{eq:vf_formula}
    \dfrac{|D|}{|\Omega|} = \dfrac{1}{C_1-C_2}
        \left(\dfrac{1}{|\Omega|}\ds\int_{\partial \Omega}\uu\cdot
        \mathbf{n}_{\Omega} \dd S - C_2\right),
\end{equation}
where $C_1$ and $C_2$ are related to $C_g$ by \eref{eq:Cs},
respectively. Since we are assuming we have complete knowledge of $\uu$
around $\partial \Omega$ from our measurement, and since $C_g = \Delta
g$ is given, \eref{eq:vf_formula} allows us to exactly determine
$|D|/|\Omega|$. Note also that we must take $C_g \ne 0$. If $C_g = 0$,
then \eref{eq:Cs} implies that $C_1 = C_2 = 0$, which makes the formula
in \eref{eq:vf_formula} undefined. We have thus proved the following
theorem.
\begin{theorem}\label{thm:main_result}
    Let $D$ and $\Omega$ be open, bounded sets in $\Rd$ ($d = 2$ or $3$)
    such that $D \subset \Omega$ and $\partial D$, $\partial \Omega$ are
    smooth. Suppose $\Rd$ is filled with a material described by the
    local elasticity tensor given by \eref{eq:my_C} and
    \eref{eq:my_lambda}. Also suppose $\ff = \nabla g$ is given and
    $\mathcal{L}_2\ff = -(\lambda_2+\mu)\nabla(\nabla\cdot\ff) -
    \mu\Delta\ff = 0$ ($\Leftrightarrow \Delta g = C_g \ne 0$) for all
    $\xx \in\Rd$. Assume that $\uu\cdot \nn_{\Omega}$ is known around
    $\partial \Omega$. Then the volume fraction of the inclusion $D$ is
    given by \eref{eq:vf_formula}.
\end{theorem}

%%%%%%%%%%%%%%%%%%%%%%%%%%%%%%%%%%%%%%%%%%%%%%%%%%%%%%%%%%%%%%%%%%%%%%%%
%%%%%%%%%%%%%%%%%%%%%%%%%%%%%%%%%%%%%%%%%%%%%%%%%%%%%%%%%%%%%%%%%%%%%%%%

\section{Finite Medium}\label{sec:finite_medium}

Consider again the linear elasticity problem from Section~\ref{sec:evf},
namely that of an inclusion $D$ in a body $\Omega$ which in turn is
embedded in an infinite medium $\Rd\setminus\overline{\Omega}$. The
isotropic and homogeneous materials in $D$ and
$\Rd\setminus\overline{D}$ have Lam\'e moduli $\lambda_1$ and
$\lambda_2$, respectively; we also assume that both materials have the
same shear modulus $\mu$. If a displacement $\ff = \nabla g$ is applied
at infinity, then the displacement $\uu = \nabla \phi$ satisfies
\eref{eq:full_displacement_problem} (so $\phi$ satisfies
\eref{eq:potential_full_displacement_problem}). Recall that we require
$\mathcal{L}_2\ff = 0$ in $\Rd$, which implies $\Delta g = C_g$ in
$\Rd$. Since $g$ and $\ff = \nabla g$ are smooth in $\mathbb{R}^d$, $g$,
$\ff$, and $\mathcal{S}_2\ff$ are continuous up to $\partial D$ from
outside $D$; in other words, the limits $g|_{\partial D^+}$,
$\ff|_{\partial D^+}$, and $(\mathcal{S}_2\ff)|_{\partial D^+}$ exist
and are finite at each point of $\partial D$, where $h|_{\partial D^+}$
and $h|_{\partial D^-}$ denote the restriction of the function $h$ to
$\partial D$ from outside and inside $D$, respectively.   

We now derive a boundary condition $\PP$ so that the solution
$\uu'$ to
\begin{equation}\label{eq:our_interior_problem}
    \cases{\mathcal{L}_1\uu' = 0 &for $\xx \in D$, \\
        \mathcal{L}_2\uu' = 0 &for $\xx \in 
            \Omega \setminus \overline{D}$, \\
        \uu', \sss'\cdot\nn_{D} = (\mathcal{S}\uu')\cdot\nn_{D} 
            &continuous across $\partial D$, \\
        \PP(\uu'_0,\trac'_0,\ff_0,\FF_0) = 0 &on $\partial \Omega$}
\end{equation}
is equal to the solution $\uu$ to \eref{eq:full_displacement_problem}
restricted to $\overline{\Omega}$, i.e., $\uu' =
\uu|_{\overline{\Omega}}$; we have defined
\begin{equation}\label{eq:boundary_data}
	\fl\uu'_0 \equiv \uu'|_{\partial \Omega^-}, \quad 
	\trac'_0 \equiv (\sss'|_{\partial \Omega^-})\cdot \nn_{\Omega},\quad 
	\ff_0 \equiv \ff|_{\partial \Omega^+}, \eqand 
	\FF_0 \equiv \left[(\mathcal{S}_2\ff)|_{\partial \Omega^+}\right]
	    \cdot\nn_{\Omega}.
\end{equation}
This allows us to apply our formula \eref{eq:vf_formula} to the problem
\eref{eq:our_interior_problem}, which is posed on the finite domain
$\Omega$. For details on a related problem (including proofs of the
well-posedness of problems similar to \eref{eq:our_interior_problem}),
see the papers of Han and Wu \cite{Han:1985:AIB,Han:1992:AEB}.    

To derive the boundary condition $\PP$, we begin by considering the
following exterior problem:
\begin{equation}\label{eq:exterior_displacement}
    \cases{\mathcal{L}_E\uut_E = 0 &for 
            $\xx \in \Rd\setminus \overline{\Omega}$,\\
        \uut_E = \uut &on $\partial \Omega$,\\
        \uut_E \rightarrow 0 &as $|\xx|\rightarrow \infty$,}
\end{equation}
where $\mathcal{L}_E\uu =
-(\lambda_E+\mu)\nabla(\nabla\cdot\uu)-\mu\Delta\uu$, $\lambda_E =
\lambda_2$, and $\uut$ is a given displacement on $\partial \Omega$.
Ultimately we wish to find the normal stress distribution
$(\ssst_E|_{\partial \Omega^+})\cdot \nn_{\Omega}$ around $\partial
\Omega$ given $\uut$ --- this mapping from the displacement on the
boundary to the traction on the boundary is defined as the exterior
Dirichlet-to-Neumann map.  
\begin{definition}\label{def:exterior_DtN}
    The \emph{exterior Dirichlet-to-Neumann (DtN) map} $\Lambda_E$ is
    defined by
    \begin{equation}\label{eq:exterior_DtN}
        \Lambda_E(\uut_E|_{\partial\Omega^+}) = 
            \Lambda_E(\uut) \equiv (\ssst_E|_{\partial\Omega^+}) \cdot 
                \nn_{\Omega} 
            = \left[(\mathcal{S}\uut_E)|_{\partial \Omega^+}\right]
                \cdot\nn_{\Omega},
    \end{equation}
    where $\uut_E$ solves \eref{eq:exterior_displacement} and $\ssst_E$
    and $\mathcal{S}\uut_E$ are given by \eref{eq:isotropic_Hooke_u}
    (with $\lambda(\xx) = \lambda_E$).
\end{definition}

%%%%%%%%%%%%%%%%%%%%%%%%%%%%%%%%%%%%%%%%%%%%%%%%%%%%%%%%%%%%%%%%%%%%%%%%

\subsection{Equivalent Boundary Value Problems}%
\label{subsec:equivalent_BVPs}

We now return to the problem \eref{eq:full_displacement_problem}, 
which has a unique solution $\uu$.  We introduce exterior fields
\begin{equation}\label{eq:exterior_fields}
	\uu_E(\xx) \equiv \uu|_{\Rd\setminus\Omega} \eqand 
	    \sss_E(\xx) \equiv \sss|_{\Rd\setminus\Omega}; 
\end{equation}
we also introduce interior fields
\begin{equation}\label{eq:interior_fields}
	\uu_I(\xx) \equiv \uu|_{\overline{\Omega}} \eqand 
	    \sss_I(\xx) \equiv \sss|_{\overline{\Omega}}. 
\end{equation}
Recall that $\lambda_E = \lambda_2$.  
\begin{lemma}\label{lem:u_tilde}
	Define $\uut_E \equiv \uu_E - \ff$ where $\uu_E$ is defined in
	\eref{eq:exterior_fields} and $\ff = \nabla g$ satisfies
	$\mathcal{L}_E\ff = 0$ in $\Rd$. Then $\uut_E$ solves
	\eref{eq:exterior_displacement} with $\uut = (\uu_I|_{\partial
	\Omega^-}) - \ff_0$, where $\ff_0 = \ff|_{\partial \Omega^+}$ is
	defined in \eref{eq:boundary_data}. 
\end{lemma}
\begin{proof}
	First, since $\mathcal{L}_E\uu_E = 0$ in
	$\Rd\setminus\overline{\Omega}$ (by
	\eref{eq:full_displacement_problem}) and $\mathcal{L}_E\ff = 0$ in
	$\Rd \setminus\overline{\Omega}$, we have
    \[
        \mathcal{L}_E\uut_E = \mathcal{L}_E(\uu_E-\ff) = 
            \mathcal{L}_E\uu_E - \mathcal{L}_E\ff = 0
    \]
    in $\Rd\setminus\overline{\Omega}$ as well. Second, recall from
    \eref{eq:exterior_displacement} that $\uut \equiv \uut_E|_{\partial
    \Omega^+} \equiv (\uu_E|_{\partial \Omega^+}) - \ff_0$. Since
    $\lambda_E = \lambda_2$, $\uu$ must be continuous across $\partial
    \Omega$, i.e., $\uu_E|_{\partial\Omega^+} = \uu_I|_{\partial
    \Omega^-}$. Hence $\uut = (\uu_I|_{\partial \Omega^-})-\ff_0$.
    Finally, $\uut_E = \uu_E - \ff \rightarrow 0$ as $|\xx|\rightarrow
    \infty$ by \eref{eq:full_displacement_problem}. Thus $\uut_E$
    solves \eref{eq:exterior_displacement} with $\uut =
    (\uu_I|_{\partial \Omega^-})-\ff_0$. 
\end{proof}

\begin{theorem}\label{thm:equivalent_bvps}
	Suppose $\uu$ solves \eref{eq:full_displacement_problem} with $\ff
	= \nabla g$ and $g = (C_g/2)\xx\cdot\xx + g_h$ where $C_g \ne
	0$ is an arbitrary constant and $\Delta g_h = 0$ in $\Rd$. Define
	$\uu_E$, $\sss_E$ and $\uu_I$, $\sss_I$ as in
	\eref{eq:exterior_fields} and \eref{eq:interior_fields},
	respectively. Finally, define $\uut_E = \uu_E - \ff$. Then $\uu_I$
	satisfies
    \begin{equation}\label{eq:interior_problem}
    	\cases{\mathcal{L}_1\uu_I = 0 &in $D$,\\
    		\mathcal{L}_2\uu_I = 0 &in $\Omega\setminus\overline{D}$,\\
    		\uu_I, \sss_I\cdot \nn_D = (\mathcal{S}\uu_I)\cdot\nn_D 
    		    &continuous across $\partial D$,\\
    		\PP\left(\uu_I|_{\partial \Omega^-},
    		    (\sss_I|_{\partial \Omega^-})
    		    \cdot\nn_{\Omega},\ff_0,\FF_0\right) = 0 
    		    &on $\partial\Omega$,}
    \end{equation}
    where
    \begin{equation}\label{eq:P_formula_main}
    \fl\PP\big(\uu_I|_{\partial \Omega^-},(\sss_I|_{\partial\Omega^-})
    	    \cdot\nn_{\Omega},\ff_0,\FF_0\big) 
    	    \equiv (\sss_I|_{\partial\Omega^-})
    	    \cdot\nn_{\Omega}
    	    -\Lambda_E\big((\uu_I|_{\partial \Omega^-})-\ff_0\big) - \FF_0
    \end{equation}
    and $\ff_0$ and $\FF_0$ are defined in \eref{eq:boundary_data}.
\end{theorem}
\begin{proof}
    By definition (see \eref{eq:full_displacement_problem} and
    \eref{eq:interior_fields}), $\uu_I$ satisfies the differential
    equations and continuity conditions in \eref{eq:interior_problem}.
    By Lemma~\ref{lem:u_tilde}, $\uut_E = \uu_E - \ff$ solves
    \eref{eq:exterior_displacement} with $\uut = (\uu_I|_{\partial
    \Omega^-}) - \ff|_{\partial \Omega +}$. By \eref{eq:exterior_DtN},
    then, we have
    \begin{equation}\label{eq:exterior_mapping}
        	(\ssst_E|_{\partial \Omega^+})\cdot\nn_{\Omega} 
        	= \Lambda_E(\uut) 
        	= \Lambda_E\big((\uu_I|_{\partial \Omega^-})-\ff_0\big).
    \end{equation}
    Since $\mathcal{S}_E$ is linear, we have
    \begin{equation}\label{eq:linear_S_E}
    	    (\ssst_E|_{\partial \Omega^+})\cdot \nn_{\Omega} 
    	    = \left[(\mathcal{S}_E\uut_E)|_{\partial \Omega^+}\right]\cdot 
    	        \nn_{\Omega}
    	    = \left[(\mathcal{S}_E\uu_E)|_{\partial \Omega^+}\right]\cdot 
    	        \nn_{\Omega} - \FF_0.
    \end{equation}
    Then \eref{eq:exterior_mapping} and \eref{eq:linear_S_E} imply
    \begin{equation}\label{eq:P_intermediate}
    	    \left[(\mathcal{S}_E\uu_E)|_{\partial\Omega^+}\right]\cdot
    	        \nn_{\Omega} 
    	    = \Lambda_E\big((\uu_I|_{\partial \Omega^-})-\ff_0\big) + \FF_0.
    \end{equation}
    Since $\lambda_E = \lambda_2$, the traction across $\partial \Omega$
    must be continuous, i.e.,
    \begin{equation*}%\label{eq:continuous_traction}
        	\left[(\mathcal{S}_E\uu_E)|_{\partial\Omega^+}\right]\cdot
        	    \nn_{\Omega} 
        	= (\sss_E|_{\partial \Omega^+})\cdot\nn_{\Omega} 
        	= (\sss_I|_{\partial \Omega^-})\cdot\nn_{\Omega}.
    \end{equation*}
    Inserting this into \eref{eq:P_intermediate} gives
    \begin{equation}\label{eq:final_bc}
    	    \Lambda_E\big((\uu_I|_{\partial \Omega^-})-\ff_0\big) + \FF_0 
    	    = (\sss_I|_{\partial\Omega^-})\cdot\nn_{\Omega}.
    \end{equation}

    We define $\PP\big(\uu_I|_{\partial \Omega^-},(\sss_I|_{\partial
    \Omega^-})\cdot\nn_{\Omega},\ff_0,\FF_0\big)$ as in
    \eref{eq:P_formula_main}. Then, due to \eref{eq:final_bc}, the
    interior part of the solution $\uu$, namely $\uu_I$, satisfies
    \eref{eq:interior_problem}.
\end{proof}

We can thus identify the solution $\uu'$ of
\eref{eq:our_interior_problem} with $\uu_I$ which solves
\eref{eq:interior_problem}, i.e., $\uu' = \uu_I =
\uu|_{\overline{\Omega}}$. In other words, the solution to
\eref{eq:our_interior_problem} in the finite domain $\Omega$ will be
exactly the same as if $\Omega$ were placed in an infinite medium with
Lam\'e parameters $\lambda_E = \lambda_2$ and $\mu$ and a displacement
$\nabla g$ were applied at infinity. Therefore, if we apply the boundary
condition 
\begin{equation}\label{eq:PP_prime}
	\PP(\uu'_0,\trac'_0,\ff_0,\FF_0) 
	= \trac'_0 - \Lambda_E(\uu'_0 - \ff_0)-\FF_0
	= 0
\end{equation}
on $\partial \Omega$, where $\uu'_0$, $\trac'_0$, $\ff_0$, and $\FF_0$
are defined in \eref{eq:boundary_data}, we can use the measurement of
$\uu'\cdot\nn_{\Omega}$ around $\partial \Omega$ (i.e.,
$\uu'_0\cdot\nn_{\Omega}$) along with \eref{eq:vf_formula} (with $\uu$
replaced by $\uu'$) to find the volume fraction occupied by $D$.  
\begin{remark}\label{rem:sigma_prime}
	Since the geometry inside the body $\Omega$ is unknown, we cannot
	write $\sss'\cdot \nn_{\Omega}$ in terms of $\uu'$ (since we would
	not know whether or not to use $\lambda_1$ or $\lambda_2$ in
	\eref{eq:isotropic_Hooke_u}). Practically, we would typically apply
	a displacement $\uu'_0$ around $\partial \Omega$ with a known $\ff$
	and measure the resulting traction $\trac'_0$ around $\partial
	\Omega$. The displacement $\uu'_0$ and traction $\trac'_0$ around
	$\partial \Omega$ must be tailored so that $\PP(\uu'_0, \trac'_0,
	\ff_0, \FF_0) = 0$ --- see \eref{eq:PP_prime}.
\end{remark}

%%%%%%%%%%%%%%%%%%%%%%%%%%%%%%%%%%%%%%%
\section{Two Dimensional Example}\label{sec:2D_example}

The results presented here were first derived in a slightly different
form by Han and Wu \cite{Han:1985:AIB,Han:1992:AEB}. We
consider the case when $d = 2$ and $\Omega$ is a disk of radius $R$
centered at the origin, denoted $B_R$. In this geometry, it is possible
to determine $\Lambda_E$ exactly by first solving
\eref{eq:exterior_displacement} for the displacement $\uut_E$ in terms
of $\uut = \uut_E|_{\partial B_R}$ and then computing the corresponding
traction around $\partial B_R$, namely 
\begin{equation*}
	(\ssst_E|_{\partial B_R^+})\cdot\nn_{B_R} = 
	    \left.\left(\ssst_E\cdot\dfrac{\xx}{R}\right)
	        \right|_{\partial B_R^+} 
	    = \left.\left[\mathcal{S}_E(\uut_E)\cdot\dfrac{\xx}{R}\right]
	        \right|_{\partial B_R^+}.
\end{equation*}  
We state the main results here; the complete calculations are given
in work by one of the authors \cite{Thaler:2014:BVI}. For more general
regions, $\Lambda_E$ may have to be computed numerically.

%%%%%%%%%%%%%%%%%%%%%%%%%%%%%%%%%%%%%%%%%%%%%%%%%%%%%%%%%%%%%%%%%%%%%%%%

\subsection{Exterior Dirichlet-to-Neumann Map}\label{subsec:ext_DtN}

We denote the polar components of $\uut_E$ by $\ut_{E,r}$ and
$\ut_{E,\theta}$. It is convenient to write
\begin{equation*}%\label{eq:polar_complex}
	\uut_E(r,\theta) = \ut_{E,r}(r,\theta) + 
	    \rmi \ut_{E,\theta}(r,\theta),
\end{equation*}
where $\rmi = \sqrt{-1}$; see the books by Muskhelishvili
\cite{Muskhelishvili:1963:SBP} and England
\cite{England:1971:CVM} for more details.  

We begin by expanding $\uut(\theta) = \ut_r(\theta) + \rmi
\ut_{\theta}(\theta)$ in a Fourier series, namely
\begin{equation}\label{eq:Fourier_series_u_main}
	\fl\utr(\theta)+\rmi\utt(\theta) = \ds\sum_{n=-\infty}^{\infty}\ut_n 
	    \rme^{\rmi n \theta}, \quad \mathrm{where} \quad  
	    \ut_n = \dfrac{1}{2\pi}\int_{0}^{2\pi} \left[\utr(\theta')+
	    \rmi\utt(\theta')\right]\rme^{-\rmi n \theta'} \dd\theta'.
\end{equation}
Then it can be shown that \cite{Thaler:2014:BVI}
\begin{equation}\label{eq:ue_main}
    \eqalign{\fl\uer(r,\theta) + \rmi\uet(r,\theta)
        = \ut_0 R r^{-1} + \ds\sum_{n=1}^{\infty} \ut_{-n}R^{n-1} 
            r^{-(n-1)}\rme^{-\rmi n\theta} \\
        + \ds\sum_{n=1}^{\infty} \left[\ut_nR^{n+1}r^{-(n+1)} + 
            \left(\frac{n-1}{\rho_E}\right)\overline{\ut_{-n}}R^{n-1}
            r^{-(n+1)}\left(r^2-R^2\right)\right]\rme^{\rmi n\theta}}
\end{equation}
for $r \ge R$ and where $\rho_E \equiv (\lambda_E+3\mu)/(\lambda+\mu)$.

Next we recall from \eref{eq:exterior_DtN} that $\Lambda_E(\uut) =
(\ssst_E|_{\partial B_R^+})\cdot\nn_{B_R}$. In polar coordinates, the
components of the traction around the boundary of the disk of radius $r
\ge R$ are $\strr(r,\theta) + \rmi \strt(r,\theta)$ (where $\strr$ is
the radial component of the traction and $\strt$ is the angular
component of the traction). In particular, the traction around $\partial
B_R$ is given by 
\begin{equation}\label{eq:Fourier_series_s_main}
	\strr(R^+,\theta) + \rmi \strt(R^+,\theta) = \Lambda_E(\uut) 
	= \ds\sum_{n=-\infty}^{\infty}\st_n\rme^{\rmi n \theta},
\end{equation}
where $\strr(R^+,\theta) + \rmi\strt(R^+,\theta) = (\strr +
\rmi\strt)|_{\partial B_R^+}$, 
\begin{equation}\label{eq:DtN_coeff_main}
    \left\{
    \eqalign{\st_n &= -\frac{2\mu}{R}(n+1)\ut_n \quad (n \ge 0),\\
        \st_{-n} &= -\frac{2\mu}{R\rho_E}(n-1)\ut_{-n}\quad (n \ge 1),}
    \right.
\end{equation}
and the coefficients $\ut_n$ are defined in
\eref{eq:Fourier_series_u_main}.

%%%%%%%%%%%%%%%%%%%%%%%%%%%%%%%%%%%%%%%%%%%%%%%%%%%%%%%%%%%%%%%%%%%%%%%%

\subsection{Nonlocal Boundary Condition}%
\label{subsec:boundary_condition}

In this section, we derive an expression for the boundary condition
$\PP(\uu'_0,\trac'_0,\ff_0,\FF_0) = 0$, where $\PP$ is defined in
\eref{eq:PP_prime}. We begin by expanding $\ff_0$ in a Fourier series
around $\partial B_R$; we have
\begin{equation}\label{eq:f_Fourier_main}
	\fl (f_r + \rmi f_{\theta})|_{\partial B_R^+} = 
	    f_r(R^+,\theta) + \rmi f_{\theta}(R^+,\theta) = 
	    f_{0,r}(\theta) + \rmi f_{0,\theta}(\theta) = 
	    \ds\sum_{n=-\infty}^{\infty} f_{0,n}\rme^{\rmi n\theta},
\end{equation}
where
\begin{equation*}
	f_{0,n} = \frac{1}{2\pi}\int_{0}^{2\pi} 
    	\left[f_{0,r}(\theta') + \rmi f_{0,\theta}(\theta')\right]
	    \rme^{-\rmi n\theta'} \dd\theta.
\end{equation*}

Next we define 
\begin{equation*}%\label{eq:F_main}
	\FF \equiv \mathcal{S}_E \ff = \mathcal{S}_E \nabla g = 
	    \lambda_E\Delta g\II + 2\mu \nabla \nabla g,
\end{equation*}
where the last equality holds by \eref{eq:S_j}. Recall from
\eref{eq:boundary_data} that $\FF_0 = (\FF|_{\partial
B_R^+})\cdot\nn_{B_R}$. In complex notation, the normal components of
$\FF$ around the boundary of a disk of radius $r > 0$ are given by
$\Frr(r,\theta) + \rmi\Frt(r,\theta)$, where $\Frr$ is the radial
component and $\Frt$ is the angular component. We can expand $(\Frr +
\rmi\Frt)|_{\partial B_R^+}$ in a Fourier series as
\begin{equation}\label{eq:F_Reit_main}
	\fl (\Frr + \rmi\Frt)|_{\partial B_R^+} = 
	    \Frr(R^+,\theta) + \rmi \Frt(R^+,\theta) =  
	    F_{0,r}(\theta) + \rmi F_{0,\theta}(\theta) =  
	    \ds\sum_{n=-\infty}^{\infty} F_{0,n}\rme^{\rmi n \theta},
\end{equation}
where
\begin{equation*}
	F_{0,n} = \frac{1}{2\pi}\int_0^{2\pi} 
	    \left[F_{0,r}(\theta') + \rmi F_{0,\theta}(\theta')\right]
	    \rme^{-\rmi n \theta'} \dd\theta'.
\end{equation*}

Next we expand $g$ in a Fourier series around the disk of radius $r > 0$
as
\begin{equation*}%\label{eq:Fourier_series_g_main}
	g(r,\theta) = \ds\sum_{n=-\infty}^{\infty}g_n(r)
	    \rme^{\rmi n \theta}, 
	\quad \mathrm{where} \quad 
	g_n(r) = \frac{1}{2\pi}\ds\int_{0}^{2\pi}g(r,\theta')
	    \rme^{-\rmi n \theta'} \dd\theta'.
\end{equation*}
We can write the Fourier coefficients $F_{0,n}$ in terms of the
coefficients
$g_n$ as 
\begin{eqnarray}\label{eq:Fn_R_main}
        \fl F_{0,n} = (\lambda_E+2\mu)
            \left.\frac{\partial^2 g_n(r)}{\partial r^2}
            \right|_{r\rightarrow R^+} + \lambda_E\left[\frac{1}{R}
            \left.\frac{\partial g_n(r)}{\partial r}
            \right|_{r\rightarrow R^+} 
            - \frac{n^2}{R^2} g_n(R^+)\right] \nonumber\\
        + 2\mu\left[-\frac{n}{R}
            \left.\frac{\partial g_n(r)}{\partial r}
            \right|_{r \rightarrow R^+}+ \frac{n}{R^2}g_n(R^+)\right].
\end{eqnarray}

Returning to \eref{eq:P_formula_main}, recall that $(\uu'|_{\partial
B_R^-}) - \ff|_{\partial B_R^+} = \uu'_0 - \ff_0 = \uut$, where $\uu'$
solves \eref{eq:our_interior_problem}. Thus if we write
\begin{equation*}%\label{eq:u_Reit_main}
	\fl (u'_r + \rmi u'_{\theta})|_{\partial B_R^-} 
	    = u'_r(R^-,\theta) + \rmi u'_{\theta}(R^-,\theta) 
	    = u'_{0,r}(\theta) + \rmi u'_{0,\theta}(\theta) 
	    = \ds\sum_{n=-\infty}^{\infty}u'_{0,n}\rme^{\rmi n\theta}, 
\end{equation*}
where
\begin{equation*}%\label{eq:u_Reit_main_coeff}
	u'_{0,n} = \frac{1}{2\pi}\int_0^{2\pi}\left[u'_{0,r}(\theta') 
	    +\rmi u'_{0,\theta}(\theta')\right]
	    \rme^{\rmi n \theta'}\dd\theta',
\end{equation*}
then $u'_{0,n} - f_{0,n} = \ut_n$ --- see
\eref{eq:Fourier_series_u_main}. The components of the traction
$(\sss'|_{\partial B_R^-})\cdot\nn_{B_R} = \trac'_0$ can be written in
polar coordinates as $t'_{0,r}(\theta) + \rmi t'_{0,\theta}(\theta)$.
This can be expanded in a Fourier series as well, namely
\begin{equation*}%\label{eq:ss_Fourier_series_main}
	\fl (\sigma'_{rr} + \rmi\sigma'_{r\theta})|_{\partial B_R^-} 
	    = \sigma'_{rr}(R^-,\theta) + \rmi\sigma'_{r\theta}(R^-,\theta) 
	    = t'_{0,r}(\theta) + \rmi t'_{0,\theta}(\theta) 
	    = \ds\sum_{n=-\infty}^{\infty} t'_{0,n} \rme^{\rmi n \theta}, 
\end{equation*}
where
\begin{equation*}%\label{eq:ss_Fourier_series_main_coeff}
	t'_{0,n} = \frac{1}{2\pi}\int_{0}^{2\pi} \left[t'_{0,r}(\theta') 
	    +\rmi t'_{0,\theta}(\theta')\right]
	    \rme^{-\rmi n\theta'}\dd\theta'.
\end{equation*}

Recalling the Fourier expansions of $\ff_0$ given in
\eref{eq:f_Fourier_main}, $\FF_0$ given in \eref{eq:F_Reit_main} (and
\eref{eq:Fn_R_main}), and $ \Lambda_E(\uut) = \Lambda_E(\uu'_0 -
\ff_0)$ given in
\eref{eq:Fourier_series_s_main}--\eref{eq:DtN_coeff_main}, the
condition $\PP(\uu'_0,\trac'_0,\ff_0,\FF_0) = 0$ is equivalent to
\begin{eqnarray}
	\trac'_0 - \FF_0 - \Lambda_E(\uu'_0-\ff_0) = 0 \nonumber \\
	\eqalign{
		\Leftrightarrow \ds\sum_{n=-1}^{\infty} 
		    \left[t'_{0,n} - F_{0,n} + \frac{2\mu}{R}(n+1)
		    \left(u'_{0,n}-f_{0,n}\right)\right]\rme^{\rmi n \theta}\\
	\qquad+ \ds\sum_{n=2}^{\infty} \left[t'_{0,-n} - F_{0,-n} + 
	    \frac{2\mu}{R\rho_E}(n-1)\left(u'_{-n}-f_{0,-n}\right)\right]
	    \rme^{-\rmi n \theta} = 0.
	}\label{eq:final_bc_main}
\end{eqnarray}
Therefore we have the following relationships between the Fourier
coefficients of the polar components of the displacement, traction, and
applied stress around $\partial B_R$:
\begin{equation}\label{eq:Fourier_relationship_main}
    \left\{
    \eqalign{t'_{0,n} - F_{0,n} + \frac{2\mu}{R}(n+1)
        \left(u'_n-f_{0,n}\right) &= 0 \quad (n \ge -1),\\
        t'_{0,-n} - F_{0,-n} + \frac{2\mu}{R\rho_E}(n-1)
            \left(u'_{-n}-f_{0,-n}\right) &= 0 \quad (n \ge 2).}
    \right.
\end{equation}
\begin{remark}\label{rem:sigma_n}
	Recall from \eref{eq:our_interior_problem} that $\trac'_0$ is the
	traction around $\partial B_R$ due to the applied displacement
	$\uu'_0$. In practice, one could consider applying a displacement
	$\uu'_0$ around $\partial \Omega$ with a known $\ff$ and then
	measuring $\trac'_0$ around $\partial B_R$. The applied displacement
	$\uu'_0$ and measured traction $\trac'_0$ have to be such that
	\eref{eq:final_bc_main} (and, hence,
	\eref{eq:Fourier_relationship_main}) holds.  
\end{remark}

%%%%%%%%%%%%%%%%%%%%%%%%%%%%%%%%%%%%%%%%%%%%%%%%%%%%%%%%%%%%%%%%%%%%%%%%

\subsection{Previous Results}\label{subsec:previous_results}

Previously, Han and Wu also derived an expression for the exterior DtN
map $\Lambda_E(\uut)$ \cite{Han:1985:AIB,Han:1992:AEB}. They found the
solution $\uut_E$ to \eref{eq:exterior_displacement} by a method
slightly different from the one we used; they then computed the
Cartesian components of the traction $\ssst_E\cdot\nn_{\Omega}$ around
$\partial B_R$. In particular, if we denote the Cartesian components of
$\uut_E$ by $\ut_E$ and $\vt_E$, the Cartesian components of
$\uut_E|_{\partial B_R^+} = \uut$ by $\ut$ and $\vt$, and the Cartesian
components of the traction $(\ssst_E|_{\partial B_R^+})\cdot\nn_{B_R}$
by $\Xt$ and $\Yt$, then $\Lambda_E(\uut) = \Lambda_E(\ut + \rmi\vt) =
\Xt + \rmi\Yt$. In particular they showed
\cite[equations~(29)~and~(30)]{Han:1992:AEB}
\begin{equation}\label{eq:exterior_DtN_disk}	
	\eqalign{\Xt = \frac{2+2\eta}{1+2\eta}\frac{\mu}{\pi R}
		    \ds\sum_{n=1}^{\infty} 
		    \int_{0}^{2\pi} 
		    \dfrac{\mathrm{d}^2\ut(\theta')}{\mathrm{d} \theta'^2}
		    \dfrac{\cos n(\theta-\theta')}{n}\dd\theta'\\
		    \qquad-
		    \frac{2\eta}{1+2\eta}\frac{\mu}{\pi R}\ds\sum_{n=1}^{\infty} 
		    \ds\int_{0}^{2\pi} 
		    \dfrac{\mathrm{d}^2\vt(\theta')}{\mathrm{d}\theta'^2}
		    \dfrac{\sin n(\theta-\theta')}{n}\dd\theta'; \\
		\Yt = \frac{2+2\eta}{1+2\eta}\frac{\mu}{\pi R}
		    \ds\sum_{n=1}^{\infty} \int_{0}^{2\pi} 
		    \dfrac{\mathrm{d}^2\vt(\theta')}{\mathrm{d} \theta'^2}
		    \dfrac{\cos n(\theta-\theta')}{n}\dd\theta' \\
		    \qquad+ 
		    \frac{2\eta}{1+2\eta}\frac{\mu}{\pi R}\ds\sum_{n=1}^{\infty} 
		    \int_{0}^{2\pi} 
		    \dfrac{\mathrm{d}^2\ut(\theta')}{\mathrm{d} \theta'^2}
		    \dfrac{\sin n(\theta-\theta')}{n}\dd\theta'}
\end{equation}
where $\ut(\theta') = \ut_E(R^+,\theta')$, $\vt(\theta') =
\vt_E(R^+,\theta')$, and $\eta = \mu/(\lambda_E+\mu)$. Also see the
books by Muskhelishvili \cite[Section~83]{Muskhelishvili:1963:SBP} and
England \cite[Section~4.2]{England:1971:CVM} for solutions to problems
related to \eref{eq:exterior_displacement} based on potential
formulations. A proof that our formulas
\eref{eq:Fourier_series_s_main}--\eref{eq:DtN_coeff_main} agree with
\eref{eq:exterior_DtN_disk} as long as $\uut$ is smooth enough is given
in the work by Thaler \cite{Thaler:2014:BVI}.

%%%%%%%%%%%%%%%%%%%%%%%%%%%%%%%%%%%%%%%%%%%%%%%%%%%%%%%%%%%%%%%%%%%%%%%%
%%%%%%%%%%%%%%%%%%%%%%%%%%%%%%%%%%%%%%%%%%%%%%%%%%%%%%%%%%%%%%%%%%%%%%%%

%\section*{Acknowledgements}
\ack

AET would like to thank Patrick Bardsley, Andrej Cherkaev, Elena
Cherkaev, Fernando Guevara Vasquez, and Hyeonbae Kang for helpful
discussions. The work of AET and GWM was supported by the National
Science Foundation through grant DMS-1211359.

%%%%%%%%%%%%%%%%%%%%%%%%%%%%%%%%%%%%%%%%%%%%%%%%%%%%%%%%%%%%%%%%%%%%%%%%
%%%%%%%%%%%%%%%%%%%%%%%%%%%%%%%%%%%%%%%%%%%%%%%%%%%%%%%%%%%%%%%%%%%%%%%%

\section*{References}
\bibliographystyle{unsrt}
\bibliography{EVF_journal_short}

\ifx \bblindex \undefined \def \bblindex #1{} \fi
\begin{thebibliography}{10}

\bibitem{McPhedran:1982:ESI}
R.~C. McPhedran, D.~R. McKenzie, and G.~W. Milton.
\newblock Extraction of structural information from measured transport
  properties of composites.
\newblock {\em Applied Physics A}, 29:19--27, 1982.

\bibitem{Phan-Thien:1982:PUB}
N.~Phan-Thien and G.~W. Milton.
\newblock A possible use of bounds on effective moduli of composite materials.
\newblock {\em Journal of Reinforced Plastics and Composites}, 1:107--114,
  1982.

\bibitem{McPhedran:1990:ITP}
R.~C. McPhedran and G.~W. Milton.
\newblock Inverse transport problems for composite media.
\newblock {\em Materials Research Society Symposium Proceedings}, 195:257--274,
  1990.

\bibitem{Kang:1997:ICP}
H.~Kang, J.~K. Seo, and D.~Sheen.
\newblock The inverse conductivity problem with one measurement: stability and
  estimation of size.
\newblock {\em SIAM Journal on Mathematical Analysis}, 28(6):1389--1405, 1997.

\bibitem{Alessandrini:1998:ICP}
G.~Alessandrini and E.~Rosset.
\newblock The inverse conductivity problem with one measurement: bounds on the
  size of the unknown object.
\newblock {\em SIAM Journal on Applied Mathematics}, 58(4):1060--1071, 1998.

\bibitem{Cherkaeva:1998:IBM}
E.~Cherkaeva and K.~Golden.
\newblock Inverse bounds for microstructural parameters of composite media
  derived from complex permittivity measurements.
\newblock {\em Waves in Random Media}, 8(4):437--450, 1998.

\bibitem{Ikehata:1998:SEI}
M.~Ikehata.
\newblock Size estimation of inclusion.
\newblock {\em Journal of Inverse and Ill-Posed Problems}, 6(2):127--140, 1998.

\bibitem{Alessandrini:1999:OSE}
G.~Alessandrini, E.~Rosset, and J.~K. Seo.
\newblock Optimal size estimates for the inverse conductivity problem with one
  measurement.
\newblock {\em Proceedings of the American Mathematical Society},
  128(1):53--64, September 2000.

\bibitem{Alessandrini:2002:DIE}
G.~Alessandrini, A.~Morassi, and E.~Rosset.
\newblock Detecting an inclusion in an elastic body by boundary measurements.
\newblock {\em SIAM Journal on Mathematical Analysis}, 33(6):1247--1268, 2002.

\bibitem{Alessandrini:2002:DCE}
G.~Alessandrini, A.~Morassi, and E.~Rosset.
\newblock Detecting cavities by electrostatic boundary measurements.
\newblock {\em Inverse Problems}, 18:1333--1353, 2002.

\bibitem{Capdeboscq:2003:OAE}
Y.~Capdeboscq and M.~S. Vogelius.
\newblock Optimal asymptotic estimates for the volume of internal
  inhomogeneities in terms of multiple boundary measurements.
\newblock {\em Mathematical Modelling and Numerical Analysis = Modelisation
  math{\'e}matique et analyse num{\'e}rique: $M^2AN$}, 37:227--240, 2003.

\bibitem{Alessandrini:2004:DIEb}
G.~Alessandrini, A.~Morassi, and E.~Rosset.
\newblock Detecting an inclusion in an elastic body by boundary measurements.
\newblock {\em SIAM Review}, 46(3):477--498, 2004.

\bibitem{Capdeboscq:2008:IHS}
Y.~Capdeboscq and H.~Kang.
\newblock Improved {Hashin}--{Shtrikman} bounds for elastic moment tensors and
  an application.
\newblock {\em Applied Mathematics and Optimization}, 57:263--288, 2008.

\bibitem{Morassi:2009:DCI}
A.~Morassi, E.~Rosset, and S.~Vessella.
\newblock Detecting general inclusions in elastic plates.
\newblock {\em Inverse Problems}, 25:045009, 2009.

\bibitem{Beretta:2011:SEE}
E.~Beretta, E.~Francini, and S.~Vessella.
\newblock Size estimates for the {EIT} problem with one measurement: the
  complex case.
\newblock {\em Revista Matematica Iberoamericana}, 2011.
\newblock Also see arxiv:1108.0052v2 [math.AP].

\bibitem{Kang:2012:SBV}
H.~Kang, E.~Kim, and G.~W. Milton.
\newblock Sharp bounds on the volume fractions of two materials in a
  two-dimensional body from electrical boundary measurements: the translation
  method.
\newblock {\em Calculus of Variations and Partial Differential Equations},
  45(3--4):367--401, 2012.

\bibitem{Milton:2012:BVF}
G.~W. Milton and L.~H. Nguyen.
\newblock Bounds on the volume fraction of 2-phase, 2-dimensional elastic
  bodies and on (stress, strain) pairs in composites.
\newblock {\em Comptes Rendus M\'ecanique}, 340(4--5):193--204, 2012.

\bibitem{Milton:2012:UBE}
G.~W. Milton.
\newblock Universal bounds on the electrical and elastic response of two-phase
  bodies and their application to bounding the volume fraction from boundary
  measurements.
\newblock {\em Journal of the Mechanics and Physics of Solids}, 60(1):139--155,
  2012.

\bibitem{Kang:2013:BSI}
H.~Kang, K.~Kim, H.~Lee, X.~Li, and G.~W. Milton.
\newblock Bounds on the size of an inclusion using the translation method for
  two-dimensional complex conductivity.
\newblock {\em SIAM Journal on Applied Mathematics}, to appear.
\newblock Also see arxiv:1310.2439 [math.AP].

\bibitem{Kang:2013:BVF3d}
H.~Kang and G.~W. Milton.
\newblock Bounds on the volume fractions of two materials in a three
  dimensional body from boundary measurements by the translation method.
\newblock {\em SIAM Journal on Applied Mathematics}, 73:475--492, 2013.

\bibitem{Thaler:2013:BVI}
A.~E. Thaler and G.~W. Milton.
\newblock Bounds on the volume of an inclusion in a body from a complex
  conductivity measurement.
\newblock {\em Communications in Mathematical Sciences}, to appear.
\newblock Also see arxiv:1306.6608 [math.AP].

\bibitem{Kang:2014:BSS}
H.~Kang, G.~W. Milton, and J.-N. Wang.
\newblock Bounds on the volume fraction of the two-phase shallow shell using
  one measurement.
\newblock {\em Journal of Elasticity}, 114(1):41--53, 2014.

\bibitem{Thaler:2014:BVI}
A.~E. Thaler.
\newblock {\em Bounds on the volume of an inclusion in a body and cloaking due
  to anomalous localized resonance}.
\newblock {Ph.D.} thesis, University of Utah, Salt Lake City, UT, 2014.

\bibitem{Hill:1963:EPR}
R.~Hill.
\newblock Elastic properties of reinforced solids: {Some} theoretical
  principles.
\newblock {\em Journal of the Mechanics and Physics of Solids}, 11(5):357--372,
  1963.

\bibitem{Hashin:1983:ACM}
Z.~Hashin.
\newblock Analysis of composite materials --- {A} survey.
\newblock {\em Journal of Applied Mechanics}, 50(3):481--505, 1983.

\bibitem{Milton:2002:TOC}
G.~W. Milton.
\newblock {\em The Theory of Composites}, volume~6 of {\em Cambridge Monographs
  on Applied and Computational Mathematics\,\,}.
\newblock Cambridge University Press, Cambridge, United Kingdom, 2002.

\bibitem{Dvorak:1997:OMI}
G.~J. Dvorak and Y.~Benveniste.
\newblock On micromechanics of inelastic and piezoelectric composites.
\newblock In T.~Tatsumi, E.~Watanabe, and T.~Kambe, editors, {\em Theoretical
  and Applied Mechanics 1996: Proceedings of the XIXth International Congress
  of Theoretical and Applied Mechanics, Kyoto, Japan, 25--31 August 1996},
  pages 217--237, Amsterdam, 1997. Elsevier.

\bibitem{Milton:1997:CMM}
G.~W. Milton.
\newblock Composites: {A} myriad of microstructure independent relations.
\newblock In T.~Tatsumi, E.~Watanabe, and T.~Kambe, editors, {\em Theoretical
  and Applied Mechanics 1996: Proceedings of the XIXth International Congress
  of Theoretical and Applied Mechanics, Kyoto, Japan, 25--31 August 1996},
  pages 443--459, Amsterdam, 1997. Elsevier.

\bibitem{Grabovsky:2000:ERE}
Y.~Grabovsky, G.~W. Milton, and D.~S. Sage.
\newblock Exact relations for effective tensors of composites: {Necessary}
  conditions and sufficient conditions.
\newblock {\em Communications on Pure and Applied Mathematics}, 53(3):300--353,
  2000.

\bibitem{Hegg:2012:ERL}
M.~Hegg.
\newblock {\em Exact relations and links for fiber-reinforced elastic
  composites}.
\newblock PhD thesis, Temple University, Philadelphia, Pennsylvania, 2012.

\bibitem{Hegg:2013:LBE}
M.~Hegg.
\newblock Links between effective tensors for fiber-reinforced elastic
  composites.
\newblock {\em Comptes Rendus M\'ecanique}, 341(6):520--532, 2013.

\bibitem{Grabovsky:1998:EREa}
Y.~Grabovsky.
\newblock Exact relations for effective tensors of polycrystals. {I}.
  {Necessary} conditions.
\newblock {\em Archive for Rational Mechanics and Analysis}, 143(4):309--329,
  1998.

\bibitem{Grabovsky:1998:EREb}
Y.~Grabovsky and D.~S. Sage.
\newblock Exact relations for effective tensors of polycrystals. {II}.
  {Applications} to elasticity and piezoelectricity.
\newblock {\em Archive for Rational Mechanics and Analysis}, 143(4):331--356,
  1998.

\bibitem{Grabovsky:1998:ERC}
Y.~Grabovsky and G.~W. Milton.
\newblock Exact relations for composites: {Towards} a complete solution.
\newblock {\em Documenta Mathematica, Journal der Deutschen
  Mathematiker-Vereinigung}, Extra Volume ICM III:623--632, 1998.

\bibitem{Grabovsky:1998:ROP}
Y.~Grabovsky and G.~W. Milton.
\newblock Rank one plus a null-{Lagrangian} is an inherited property of
  two-dimensional compliance tensors under homogenization.
\newblock {\em Proceedings of the Royal Society of Edinburgh},
  128A(2):283--299, 1998.

\bibitem{Hill:1964:TMP}
R.~Hill.
\newblock Theory of mechanical properties of fibre-strengthened materials: {I}.
  {Elastic} behaviour.
\newblock {\em Journal of the Mechanics and Physics of Solids}, 12(4):199--212,
  1964.

\bibitem{Givoli:1991:NRB}
D.~Givoli.
\newblock Non-reflecting boundary conditions.
\newblock {\em Journal of Computational Physics}, 94(1):1--29, 1991.

\bibitem{Han:1985:AIB}
H.~Han and X.~Wu.
\newblock Approximation of infinite boundary condition and its application to
  finite element methods.
\newblock {\em Journal of Computational Mathematics}, 3(2):179--192, 1985.

\bibitem{Han:1992:AEB}
H.~Han and X.~Wu.
\newblock The approximation of the exact boundary conditions at an artificial
  boundary for linear elastic equations and its application.
\newblock {\em Mathematics of Computation}, 59(199):21--37, 1992.

\bibitem{Lee:2006:EAB}
S.~Lee, R.~E. Caflisch, and Y.-J. Lee.
\newblock Exact artificial boundary conditions for continuum and discrete
  elasticity.
\newblock {\em SIAM Journal on Applied Mathematics}, 66(5):1749--1775, 2006.

\bibitem{Bonnaillie-Noel:2013:ACL}
V.~Bonnaillie-No{\"e}l, M.~Dambrine, F.~H{\'e}rau, and G.~Vial.
\newblock Artificial conditions for the linear elasticity equations.
\newblock {\em Mathematics of Computation}, to appear.

\bibitem{Ammari:2007:PMT}
H.~Ammari and H.~Kang.
\newblock {\em Polarization and Moment Tensors With Applications to Inverse
  Problems and Effective Medium Theory}, volume 162 of {\em Applied
  Mathematical Sciences}.
\newblock Springer, New York, 2007.

\bibitem{Hill:1961:DRM}
R.~Hill.
\newblock Discontinuity relations in mechanics of solids.
\newblock In I.~N. Sneddon and R.~Hill, editors, {\em Progress in Solid
  Mechanics}, volume~{II}, chapter~{VI}, pages 247--276. North-Holland
  Publishing Co., Amsterdam, 1961.

\bibitem{Evans:2010:PDE}
L.~C. Evans.
\newblock {\em Partial Differential Equations}, volume~19 of {\em Graduate
  Studies in Mathematics\,\,}.
\newblock American Mathematical Society, Providence, 2010.

\bibitem{Muskhelishvili:1963:SBP}
N.~I. Muskhelishvili.
\newblock {\em Some Basic Problems of the Mathematical Theory of Elasticity:
  Fundamental Equations, Plane Theory of Elasticity, Torsion, and Bending}.
\newblock P. Noordhoff, Groningen, The Netherlands, 1963.

\bibitem{England:1971:CVM}
A.~H. England.
\newblock {\em Complex Variable Methods in Elasticity}.
\newblock Wiley-In{\-}ter{\-}sci{\-}ence, New York, 1971.

\end{thebibliography}

\end{document}